\documentclass[11pt,a4paper]{article}

\usepackage{graphicx}
\usepackage{amssymb,amsmath,amsthm}
\usepackage{pdfsync}
\usepackage[latin1]{inputenc}
\usepackage{enumerate}
\usepackage{amsfonts}
\usepackage{amssymb}
\usepackage{hyperref}

 \newtheorem{thm}{Theorem}[section]
 \newtheorem{cor}[thm]{Corollary}

 \theoremstyle{definition}
 
 \theoremstyle{remark}
 \newtheorem{rem}[]{Remark}
 
 \numberwithin{equation}{section}

\def\N{\mathbb N}

\def\R{\mathbb R}

\begin{document}

\title{On some relations between the  Perimeter, the Area and the Visual Angle of a Convex Set}

\author{J.Bruna, J.Cufí and  A.Reventós}

\date{}


%

\maketitle

\begin{abstract}We establish some relations between the perimeter, the area  and the visual angle of a planar compact convex set. Our first result  states that Crofton's formula is the unique universal formula relating the visual angle, length and area. After that  we give a characterization of convex sets of constant width by means of the behaviour of its isotopic sets at infinity. Also for this class of convex sets we prove that the existence of an isotopic circle is enough to ensure that the considered set is a disc. 
\end{abstract}

\section{Introduction}
In this paper we establish some relations between geometric quantities associated to a planar compact convex 
set $K$ with boundary of class ${\cal C}^{2}$. More precisely we consider the perimeter $L$, the area $F$ of $K$ and the visual angle $w=w(P)$ of $K$ from a point $P\notin K$.

The starting point is the classical Crofton's formula 

\begin{eqnarray}\label{0103}\int_{P\notin K}(w-\sin w)dP=\frac{L^{2}}{2}-\pi F.\end{eqnarray} 
This equality easily follows from standard arguments of Integral Geometry, see for instance \cite{santalo}. Another approach to \eqref{0103} is given in \cite{CGR}. The natural question arises whether separate formulas for $L^2, F$ alone exist, or equivalently, whether replacing $w-\sin w$ by some other function $f(w)$ one gets a different linear combination of $L^2,F$. For specific domains this is indeed possible, for instance when $K$ is a disc  one has several formulas of this type, like 

$$L^{2}=\int_{P\notin K}\frac{4}{3}\sin^{3}w\,dP,\qquad 4\pi F=\int_{P\notin K}\frac{\sin^{3}w}{\cos^{2}(w/2)}\,dP,$$
see \cite{santalo}, p.59.

Our first result (Theorem \ref{2102}) states, under suitable conditions,  that if a function $f$ satisfies 
 $$\int_{P\notin K}f(\omega(P))dP=aL^{2}+bF$$ for each compact convex set $K$, with $a,b\in\R$, two constants that do not depend on $K$, 
then $$f(w)=\lambda(w-\sin w),$$
for some constant $\lambda\in\R$. A consequence of this result is that no universal formula exists giving say the area $F$ in terms of the visual angle alone. We believe, though, that a  formula like
$$F=\int_{P\notin K} f(\omega(P))g(P)\, dP,$$
with $f,g$  universal functions not depending on $K$ might exist.

\medskip

By the result in \cite{Kurusa}, $K$ is completely determined by the visual angle $w(P)$ outside a big ball containing $K$. It is then natural to ask how specific properties of $K$ can be read from this knowledge. 

In our next results, we do so in terms of the assymptotic behaviour of $w(P)$ as $P$ goes to infinity. Now,  it is easy to see that, denoting by  $w(R,\theta)$  the visual angle of $K$ from the point of polar coordinates $(R,\theta)$, the quantity $R\,w(R,\theta)$ remains bounded and

 \begin{equation}\label{aa}
 \lim_{R\to\infty} R\,w(R,\theta)=a(\theta),
 \end{equation}
  where $a(\theta)$ is the width of $K$ in the direction $\theta$, meaning that the knowledge of the asymptotic behaviour of $w$ at infinity amounts to know the width function. Whence, assymptotic statements can just involve quantities that depend only on the width function. For instance, the perimeter of the convex set, that is related to the width by the formula 
  $2L=\int_{0}^{2\pi}a(\theta)\,d\theta,$ 
can be known 
  from the behaviour of the visual angle at infinity 
by 
 \begin{equation}\label{ab}
 2L=\lim_{R\to\infty}R\int_{0}^{2\pi}w(R,\theta)\,d\theta,
 \end{equation}
 see Theorem \ref{2202} and Remark \ref{1903}.
 
 For central symmetric convex sets the area is determined by the width, and this enable us to obtain the following formula 
  in the spirit of \eqref{ab},
\begin{equation}\label{ac} 
\lim_{R\to\infty}\int_{0}^{2\pi}R^{2}[w(R,\theta)^2-w_{\theta}(R,\theta)^2]\,d\theta= 8F,
\end{equation}
where $w_{\theta}$ means partial derivative of $w$ with respect to $\theta$ (see   Corollary \ref{0103c}  and Remark \ref{1903}).

%
%
%
%
%
%

 \medskip
 After that we deal with compact convex sets of constant width for which we obtain two results. To state the first one we note that formula \eqref{aa} says that the convex set  $K$ has constant width $a$ if and only if the visual angle  $w$ behaves like $ a/R$ at infinity, which in turn roughly says that the isotopic sets $C(\alpha)=\{P:w(P)=\alpha\}$ behave like the circles of radius $R=\frac {a}{\alpha}$. Our result provides a precise quantitative formulation of this fact by establishing that $C(\alpha)$ tends to a circle in the sense that it tends to satisfy the equality in the isoperimetric inequality. More precisely, if $L(\alpha), F(\alpha)$ denote respectively the length and area of $C(\alpha)$, one has

   $$\lim_{\alpha\to 0}\frac{L(\alpha)^{2}}{4\pi F_{\alpha}}=1,$$
if and only if $K$ has constant width, see Theorem \ref{0102b}.
 
The second result deals with isotopic sets that are actually circles, not just assymtotically. Green's  (\cite{green}) proved  that if $K$ has an isotopic circle  $C_{\alpha}$, 
with $\alpha$ an irrational multiple of $\pi$, or $\pi-\alpha=(m/n)\pi$, $m$ even, $m$ and $n$ relatively prime, then $K$ is a disc. In general $K$ can have an isotopic circle without being a disc. Later on Nitsche \cite{Nitsche} proved that if $K$ has two concentric isotopic circles, then $K$ is a disc. We prove in Theorem \ref{0403} that if $K$ has  constant width and just one  isotopic circle, then it is a disc.

 \medskip
 At the end we give the relationship between  the area $F$ and the perimeter $L$  of a convex set and the radius of an isotopic circle, see \eqref{0403f} and \eqref{0403e}. As a consequence we obtain the inequalities 
 $$F\leq \pi R^{2} \sin^{2}(\frac{\alpha}{2}),\quad L\leq 2\pi R\sin (\frac{\alpha}{2}),$$
where  $R$ is the radius of the isotopic circle and $\alpha$ the constant visual angle on this circle.

 \section{On Crofton's Function}
As we have said in the Introduction, the Crofton's formula \eqref{0103} is proven with standard arguments of integral geometry, but it is also a direct consequence of the general formula for integrating functions of the visual angle given in \cite{CGR}. Concretely, equation (16) in \cite{CGR}, says that for a differentiable function $f:[0,\pi]\rightarrow \R$, satisfying  $f(w)=O(w^{3}),$ $w\to 0$,  one has
\begin{eqnarray}\label{set23x}
\int_{P\notin K}f(\omega)dP&=&
-f(\pi)F+ \frac{L^{2}}{2\pi}M(f)\nonumber\\ \notag
&+& 
\pi\sum_{k\geq 2, even}c_{k}^{2}\left(M(f) 
+ 2\sum_{j=1, odd}^{k-1}\alpha_{j} \right)
\\ 
&+& 
\pi\sum_{k\geq 3, odd}c_{k}^{2}\left( -2\sum_{j=2, even}^{k-1}\alpha_{j} \right),
\end{eqnarray}
with  $$\alpha_{j}=\int_{0}^{\pi}f'(\omega)j\cos(j\omega) d\omega ,\qquad M(f)=\int_{0}^{\pi}\frac{f'(\omega)}{1-\cos\omega}d\omega,$$ 
and $c_{k}^{2}=a_{k}^{2}+b_{k}^{2}$, where $a_{k},b_{k}$ are the Fourier coeficients of the support function  $p(\varphi)$ of the  compact convex set $K$. Notice that up to a constant, $\alpha_j$ is the $j$-th Fourier coefficient of $f'$ in the basis $\{\cos jw\}$.  The above  equality for $f(w)=w-\sin w$ gives immediately Crofton's formula.

\medskip

We shall prove now that the function $w-\sin w$  
 is the only one 
 that can provide a Crofton's type formula.

 \medskip

 \begin{thm}\label{2102}Let $f:[0,\pi]\longrightarrow \R$ be a differentiable function with $f(w)=O(w^{3}),$ $w\to 0$, such that for every compact convex set $K$, 
\begin{eqnarray}\label{1701}
\int_{P\notin K}f(\omega(P))dP=aL^{2}+bF, 
\end{eqnarray}
being $a,b\in\R$ some constants not depending on $K$,  and $w(P)$  the visual angle of $K$ from the point $P$. Then $f$ 
is, up to a constant factor $\lambda$,  the Crofton function $f(\omega)=\omega-\sin\omega$. In this case $a=\lambda/2$, $b=-\pi\lambda$.
\end{thm}

\begin{proof}
We consider the family of convex sets given by the support functions
\begin{eqnarray}\label{2002}p(\varphi)=1+t\cos(m\varphi),\qquad 0\leq\varphi\leq 2\pi,\end{eqnarray} for $m\in\N$.
The condition of convexity $p+p''>0$ is satisfied if $0<t<\frac{1}{m^{2}-1}$.

Then, the perimeter $L$ and the area $F$ of these convex sets are
\begin{eqnarray*}
L=\int_{0}^{2\pi}p\,d\varphi=2\pi,\qquad F=\frac{1}{2}\int_{0}^{2\pi}(p^{2}-p'^{2})\,d\varphi=\pi-\frac{\pi}{2}(m^{2}-1)t^{2}.
\end{eqnarray*}

Now we combine \eqref{set23x} and \eqref{1701} to obtain, with $m$ even,

\begin{eqnarray*}a4\pi^{2}+b(\pi-\frac{\pi}{2}(m^{2}-1)t^{2})&=&-f(\pi)(\pi-\frac{\pi}{2}(m^{2}-1)t^{2})+2\pi M(f)\\&+&\pi M(f)t^{2}+2\pi (\alpha_{1}+\alpha_{3}+\dots+\alpha_{m-1})t^{2}.\end{eqnarray*}
Equalizing the coefficients of $t^{2}$ we obtain 
$$(m^{2}-1)(b+f(\pi))+2M(f)+4(\alpha_{1}+\alpha_{3}+\dots+\alpha_{m-1})=0.$$

Changing $m$ by $m+2$ it follows 
$$((m+2)^{2}-1)(b+f(\pi))+2M(f)+4(\alpha_{1}+\alpha_{3}+\dots+\alpha_{m-1}+\alpha_{m+1})=0.$$
 Substracting these two last equalities
we get for $m\geq 2$, $m$ even, 
$$b+f(\pi)=-\frac{\alpha_{m+1}}{m+1}=-\int_{0}^{\pi}f'(w)\cos((m+1)w)\,dw.$$

From Riemann-Lebesgue's Lemma it follows that 
$b+f(\pi)=0$, and $\alpha_{j}=0$, for $j$ odd, $j\geq 3$.

\medskip
As well, for  $m$ odd,  we have  from \eqref{set23x}
\begin{eqnarray*}a4\pi^{2}+b(\pi-\frac{\pi}{2}(m^{2}-1)t^{2})&=&-f(\pi)(\pi-\frac{\pi}{2}(m^{2}-1)t^{2})+2\pi M(f)\\&-&2\pi (\alpha_{2}+\alpha_{4}+\dots+\alpha_{m-1})t^{2}.\end{eqnarray*}
Equalizing the coefficients of $t^{2}$ and using that $b+f(\pi)=0$ we obtain
$$\alpha_{2}+\alpha_{4}+\dots+\alpha_{m-1}=0, \quad m\geq 3$$
and so $\alpha_{j}=0$, for $j$ even, $j\geq 2$.

%
%
%
%
%
%


Hence, 
$$f'(\omega)=a_{0}+a_{1}\cos(\omega)$$
  and 
$$f(\omega)=a_{0}\,\omega+a_{1}\sin(\omega)+ c,$$
for some constants $a_{0}, a_{1},c$.

But since  $f(w)=O(w^{3}),$ $w\to 0$, we get  $f(0)=f'(0)=0$ and so $c=0, \, a_{0}+a_{1}=0$. This proves the Theorem with $\lambda=a_{0}$. 

\end{proof}


\section{Behaviour of the visual angle at infinity}
The goal of this section is to obtain information about the convex set observing it from a point that goes to infinity.

First of all we will see that the perimeter of a  convex set can be evaluated  by integrating the visual angle on circles of increasing radius. 

The circle $C_{R}$ centered at the origin with radius $R$, can  be parametrized by means of the support function $p(\varphi)$ of $K$ in the following way. To each value of $\varphi$ one associates the point $P(R,\varphi)$ given by the intersection of $C_{R}$ with the  half-straight line, taken in the direct sense,  of slope $\varphi+\pi/2$,  and at a distance $p(\varphi)$ of the origin.


 Then we have
\begin{thm}\label{2202}Let $K$ be a compact convex set of perimeter $L$ and  denote by $w(R,\varphi)$ the visual angle of $K$ from the point  $P(R,\varphi)$. Then
$$2L=\lim_{R\to\infty}R\int_{0}^{2\pi}w(R,\varphi)\,d\varphi.$$
\end{thm}
\begin{proof}The visual angle $w=w(R,\varphi)$ verifies the fundamental relation
$$\arccos\frac{p(\varphi)}{R}+\arccos\frac{p(\varphi+\pi-w)}{R}=\pi-w(R,\varphi),\quad  0\leq\varphi\leq 2\pi,$$
for every $R>0$ such that $C_{R}$ contains $K$, where $p(\varphi)$ is the support function of $K$.

From this equation it follows
\begin{eqnarray}\label{2202c}
p^{2}+p_{1}^{2}+2pp_{1}\cos(w)=R^{2}\sin^{2}(w), 
\end{eqnarray}
where $p=p(\varphi)$, $p_{1}=p(\varphi+\pi-w)$.

Then from 
 \eqref{2202c} we have that $\lim_{R\to\infty}w(R,\varphi)=0$ and 
\begin{eqnarray}\label{2202d}\lim_{R\to\infty}R\, w(R,\varphi)=\lim_{R\to\infty}R\sin(w(R,\varphi))=a(\varphi), \quad 0\leq\varphi\leq 2\pi,\end{eqnarray} where $a(\varphi)=p(\varphi)+p(\varphi+\pi)$
is the width of $K$ in the direction $\varphi$. 

The  limit in \eqref{2202d} is uniform in $\varphi$. In fact, $w(R,\varphi)\leq \tilde{w}(R)$ where  $\tilde{w}(R)$ is the visual angle of the smallest circle centered at the origin and containing $K$ from a point at distance $R$ from the origin. So $w(R,\varphi)$ tends to zero uniformly in $\varphi$ when $R\to \infty$ and we deduce,  from \eqref{2202c} and the uniform continuity of $p(\varphi)$,  
that the convergence in \eqref{2202d} is uniform.

Then the result follows by integration in \eqref{2202d}.
\end{proof}

Motivated  by Theorem \ref{2202} we can ask if there is an analogous result involving the area of $K$. 
We can answer this question for central  symmetric compact convex sets. The basic result  is

\begin{thm}\label{2202m}Let $K$ be a compact convex set and let $w=w(R,\varphi)$ be the visual angle of $K$ at the point $P(R,\varphi)$. Denote by $a(\varphi)=p(\varphi)+p(\varphi+\pi)$, where $p(\varphi)$  is the support function of $K$, the width of $K$ in the direction $\varphi$.
Then 

$$\lim_{R\to\infty}\int_{0}^{2\pi}R^{2}(w(R,\varphi)^{2}-w_{\varphi}(R,\varphi)^{2})\,d\varphi=\int_{0}^{2\pi}(a(\varphi)^{2}-a'(\varphi)^{2})\,d\varphi,$$
where  $w_{\varphi}$ denotes the derivative with respect to  $\varphi$.
\end{thm}

\begin{proof}

%

We begin by proving that 
\begin{eqnarray}\label{2202h}\lim_{R\to\infty} R\,w_{\varphi}(R,\varphi)=a'(\varphi),\end{eqnarray}
uniformly on $\varphi$. In fact, differentiation of equation \eqref{2202c} with respect to $\varphi$ gives $$Rw_{\varphi}=\frac{R(pp'+p_{1}p'_{1}+(p'p_{1}+pp'_{1})\cos(w))}{R^{2}\sin(w)\cos(w)+p_{1}p'_{1}+pp'_{1}+pp_{1}}.$$

Taking limits and according \eqref{2202d} we have 

$$\lim_{R\to\infty}Rw_{\varphi}(R,\varphi)=\lim_{R\to\infty}\frac{Ra(\varphi)a'(\varphi)}{Ra(\varphi)+p'(\varphi+\pi)a(\varphi)+p(\varphi)p(\varphi+\pi)}=a'(\varphi).$$
Since the convergence in \eqref{2202d} is uniform and the functions $p(\varphi)$ and  $p'(\varphi)$ are uniformly continuous,  
the convergence in \eqref{2202h} is also uniform.

As a consequence 
$$\lim_{R\to \infty}R^{2}\int_{0}^{2\pi}(w(R,\varphi)^{2}-w_{\varphi}(R,\varphi)^{2})\,d\varphi=\int_{0}^{2\pi}(a(\varphi)^{2}-a'(\varphi)^{2})\,d\varphi,$$
as we wanted to prove. 
\end{proof}

We note that the integral in the right-hand side of the above equality is two times the area of the convex set having $a(\varphi)$ as its support function.
 
 \medskip
 As a consequence of the above Theorem we obtain for the especial case of central symmetric  convex sets
 the following result. 
 
\begin{cor}\label{0103c}Let $K$ be a compact convex set symmetric with respect to the origin and of area $F$. With the notation in  Theorem \ref{2202m} one has $$F=\frac{1}{8}\lim_{R\to\infty}\int_{0}^{2\pi}R^{2}(w(R,\varphi)^{2}-w_{\varphi}(R,\varphi)^{2})\,d\varphi.$$
\end{cor}
\begin{proof} By the central symmetry of $K$ one has
$p(\varphi)=p(\varphi+\pi)$, $0\leq\varphi\leq 2\pi$, and so 
 $a(\varphi)=2p(\varphi)$ and so $a^{2}-a'^{2}=4(p^{2}-p'^{2})$ that, integrating with respect to $\varphi$ and according Theorem \ref{2202m}, gives the result.
\end{proof}

\begin{rem}\label{1903}
If we denote by $w(R,\theta)$ the visual angle of $K$ from the point of polar coordinates $(R,\theta)$
we can consider $$\int_{0}^{2\pi}Rw(R,\theta)\,d\theta$$
which is in general diferent from
$$\int_{0}^{2\pi}Rw(R,\varphi)\,d\varphi.$$
The relation between 
$\theta$ and $\varphi$ is $\theta=\varphi+\arccos\frac{p(\varphi)}{R}$.


So one has 
\begin{eqnarray*}\int_{0}^{2\pi}R\omega(R,\theta)\,d\theta&=&\int_{0}^{2\pi}Rw(R,\varphi)(1-\frac{p'(\varphi)}{\sqrt{R^{2}-p(\varphi)^{2}}})d\varphi.
\end{eqnarray*}
As a consequence, Theorem \ref{2202}  gives 

$$2L=\lim_{R\to \infty}\int_{0}^{2\pi}Rw(R,\theta)d\theta,$$
as stated in the introduction and analogously Corollary \ref{0103c} gives 
\begin{equation*} 
\lim_{R\to\infty}\int_{0}^{2\pi}R^{2}[w(R,\theta)^2-w_{\theta}(R,\theta)^2]\,d\theta= 8F,
\end{equation*}
so that in these results  
 we can use both polar coordinates $(R,\theta)$ or the coordinates $(R,\varphi)$ associated to the convex set.

\end{rem}

\section{A characterization  of convex sets of constant width by means of isotopic sets}
Given a compact convex set $K$ we denote by $C_{\alpha}$ the (isotopic) set of points in the plane from which $K$ is seen with angle $\alpha$. 
In view of the isoperimetric inequality we will say that the sets $C_{\alpha}$ tend to a circle, as $\alpha\to 0$, if $$\lim_{\alpha\to 0}\frac{L(\alpha)^{2}}{4\pi F_{\alpha}}=1,$$
where $L(\alpha)$ is the length  of $C_{\alpha}$ and $F(\alpha)$ is  the area enclosed by $C_{\alpha}$.

\begin{thm}\label{0102b}
Let $K$ be a compact convex set. Then the isotopic sets $C_{\alpha}$ of $K$ tend, as $\alpha\to 0$, to a circle if and only if $K$ is of constant width.
\end{thm}

\begin{proof} It is known, see for instance \cite{CGR}, that the points $(X,Y)$ in  $C_{\alpha}$ can be parametrized by $\varphi$  as follows
 \begin{eqnarray*}
 X&=&-\frac{1}{\sin \alpha}(p\sin(\varphi-\alpha)+p_{1}\sin\varphi)\\
  Y&=&\frac{1}{\sin \alpha}(p\cos(\varphi-\alpha)+p_{1}\cos\varphi)
 \end{eqnarray*}
with $p=p(\varphi)$ the support function of $K$,
 and $p_{1}=p(\varphi+\pi-\alpha)$.

Hence, the length $L(\alpha)$ of $C_{\alpha}$ is given by
$$L(\alpha)=\int_{0}^{2\pi}\sqrt{X'^{2}+Y'^{2}}\,d\varphi.$$

 A direct computation shows that
\begin{eqnarray}\label{1711}L(\alpha)=\frac{1}{\sin(\alpha)}\int_{0}^{2\pi}\sqrt{\Delta(\varphi,\alpha)}\,d\varphi\end{eqnarray}
with
$$\Delta(\varphi,\alpha)=p^{2}+p_{1}^{2}+p'^{2}+p_{1}'^{2}+ 2(pp_{1}+p'p_{1}')\cos(\alpha)+2(pp_{1}'-p'p_{1})\sin(\alpha).$$
So we have 
 \begin{eqnarray*}\lim_{\alpha\rightarrow 0}L(\alpha)\sin(\alpha)&=&\int_{0}^{2\pi}\sqrt{(p(\varphi)+p(\varphi+\pi))^{2}+(p'(\varphi)+p'(\varphi+\pi))^{2})}\,d\varphi\\&=&\int_{0}^{2\pi}\sqrt{a(\varphi)^{2}+a'(\varphi)^{2}}\,d\varphi,\end{eqnarray*}
where $a(\varphi)$ is the width of $K$ in the direction $\varphi$.

On the other hand, the area $F(\alpha)$ enclosed by $C_{\alpha}$ satisfies   (see \cite{CGR})

$$\lim_{\alpha\to 0}(F(\alpha)\sin^{2}\alpha)=\frac{L^{2}}{\pi}+2\pi\sum_{k\geq 2, even}c_{k}^{2}.$$

Thus
\begin{eqnarray*}\lim_{\alpha\to 0}\frac{L(\alpha)^{2}}{4\pi F(\alpha)}=\lim_{\alpha\to 0}\frac{L(\alpha)^{2}\sin^{2}(\alpha)}{4\pi F(\alpha)\sin^{2}(\alpha)}=\frac{[\int_{0}^{2\pi}\sqrt{a(\varphi)^{2}+a'(\varphi)^{2}}\,d\varphi]^{2}}{4L^{2}+8\pi^{2}\sum_{k\,even}c_{k}^{2}}.\end{eqnarray*}

If $K$ is a convex set of constant width $a$, $L=\pi a$ and  $c_{k}=0$ for $k$ even. Hence 
\begin{eqnarray*}\lim_{\alpha\to 0}\frac{L(\alpha)^{2}}{4\pi F(\alpha)}=1,\end{eqnarray*}
that proves one of the implications of the theorem.

Before looking at the converse, let us  check that the width $a(\varphi)$ satisfies
$$2\pi\int_{0}^{2\pi}a(\varphi)^{2}\,d\varphi=4L^{2}+8\pi^{2}\sum_{k\, even}c_{k}^{2}.$$
Indeed,
\begin{eqnarray*}
2\pi\int_{0}^{2\pi}a(\varphi)^{2}\,d\varphi&=&4\pi\int_{0}^{2\pi}p(\varphi)^{2}\,d\varphi+4\pi\int_{0}^{2\pi}p(\varphi)p(\varphi+\pi)\,d\varphi\\&=&
4\pi(2\pi a_{0}^{2}+\pi\sum_{k}c_{k}^{2})+4\pi(2\pi a_{0}^{2}+\pi\sum_{k}(-1)^{k} c_{k}^{2})
\\&=&4L^{2}+8\pi^{2}\sum_{k\, even}c_{k}^{2}.
\end{eqnarray*}

Hence $$\lim_{\alpha\to 0}\frac{L(\alpha)^{2}}{4\pi F(\alpha)}=\frac{[\int_{0}^{2\pi}\sqrt{a^{2}+a'^{2}}\,d\varphi]^{2}}{2\pi\int_{0}^{2\pi}a^{2}\, d\varphi}.$$

Now assuming that the above limit is equal to $1$ we have 
 \begin{eqnarray}\label{3011b}\left[\int_{0}^{2\pi}\sqrt{a^{2}+a'^{2}}\,d\varphi\right]^{2}=2\pi\int_{0}^{2\pi}a^{2}\,d\varphi.\end{eqnarray}

But this equality implies  that $a(\varphi)$ is constant. In fact \eqref{3011b} says that equality holds in the isoperimetric inequality applied to the curve given in polar coordinates by $r=a(\varphi)$.

\end{proof}

\section{Isotopic circles}
In this section we consider  the particular case in which the isotopic set  $C_{\alpha}$ of a compact convex set $K$ is a circle. 
We will say that $C_{\alpha}$ is an isotopic circle of $K$.

It is known that if a compact convex set $K$ has two concentric isotopic circles, then $K$ is a disc, see \cite{Nitsche}. The existence of only one isotopic circle is not enought to conclude that $K$ is a disc, for instance all the ellipses have an  isotopic circle with $\alpha=\pi/2$, see \cite{green}. 

In fact we can provide a family of compact convex sets having an isotopic circle with visual angle $\alpha=\pi/2$ and different from discs or ellipses. The examples given in \cite{green} do not have this property. 

To construct this family we remark that, from \eqref{2202c}, 
it follows that $K$ has an isotopic circle $C_{\pi/2}$ of radius $R$ if
$$p(\varphi)^{2}+p(\varphi+\pi/2)=R^{2}.$$

If we write the Fourier series for the function $p(\varphi)^{2}$ as 
$$p(\varphi)^{2}=\sum_{-\infty}^{\infty}c_{k}e^{ik\varphi},$$
it follows 
 $$p(\varphi)^{2}+p(\varphi+\pi/2)^{2}=\sum_{-\infty}^{\infty}c_{n}e^{in\varphi}(1+e^{in\pi/2}),$$
so that this quantity is constant  if and only if   
 $$c_{k}=0, \quad  k\neq 2+4m,\quad  m \mbox{  integer.}$$
 
 Then, any positive $2\pi$-periodic function   $p(\varphi)$ with $p+p''>0$, such that the Fourier series of $p(\varphi)^{2}$ has only coefficients $c_{k}$ with $k$ congruent to  $2$ module $4$, 
will give rise to a convex set seen from angle $\pi/2$ from a circle. 
For instance take 

$$p(\varphi)=\sqrt{15+9\cos^{2}(\varphi)+4\sin^{2}(\varphi)+\cos(6\varphi)}.$$

%

\subsection{Convex sets  of constant width with an isotopic circle}
For compact convex sets of constant  width the above quoted result  in \cite{Nitsche}, about convex sets with two isotopic circles,  can be improved. Concretely we have 
\begin{thm}\label{0403}
If a compact convex set $K$ of constant width has an isotopic circle, then  $K$ is a disc.
\end{thm}
\begin{proof}Assume that $K$ has an isotopic circle of radius $R$ with visual angle $\alpha$. Then 
equation \eqref{2202c} applied with angle $\varphi+\pi$ instead of $\varphi$ gives
$$p(\varphi+\pi)^{2}+p(\varphi-\alpha)^{2}+2p(\varphi+\pi)p(\varphi-\alpha)\cos(\alpha)=C,$$
with $C$ some constant.
By the condition of constant width, 
$p(\varphi)+p(\varphi+\pi)=a$, one has $$(a-p(\varphi))^{2}+p(\varphi-\alpha)^{2}+2(a-p(\varphi))p(\varphi-\alpha)\cos(\alpha)=C.$$

Changing the constant one  can write 
$$p(\varphi)^{2}-2ap(\varphi)+p(\varphi-\alpha)^{2}+2ap(\varphi-\alpha)\cos(\alpha)-2p(\varphi)p(\varphi-\alpha)\cos(\alpha)=C.$$
Replacing $\varphi$ by $\varphi-\alpha$ and taking into account that  $p(\varphi)$ is $2\alpha$-periodic (see \cite{green}) it follows  
$$p(\varphi-\alpha)^{2}-2ap(\varphi-\alpha)+p(\varphi)^{2}+2ap(\varphi)\cos(\alpha)-2p(\varphi-\alpha)p(\varphi)\cos(\alpha)=C.$$

Substracting the last two equalities  it follows
$$2a\bigg(p(\varphi-\alpha)-p(\varphi)+(p(\varphi-\alpha)-p(\varphi))\cos(\alpha)\bigg)=0$$  
and so 
$$p(\varphi)=p(\varphi-\alpha),$$
that is, $p(\varphi)$ 
is $\alpha$-periodic.

Then equation \eqref{2202c} reads 

$$p(\varphi)^{2}+p(\varphi+\pi)^{2}+2p(\varphi)p(\varphi+\pi)\cos(\alpha)=C,$$
that together with $p(\varphi)+p(\varphi+\pi)=a$ gives 

$$p(\varphi)^{2}+p(\varphi)^{2}-2ap(\varphi)+2ap(\varphi)\cos(\alpha)-2p(\varphi)p(\varphi)\cos(\alpha)=C,$$
or
$$(2p(\varphi)^{2}-2ap(\varphi))(1-\cos(\alpha))=C.$$
In conclusion  $p(\varphi)$ is,  for $0\leq\varphi\leq 2\pi$, a solution  of a second degree equation  $x^{2}+mx+n=0$, with  $m,n \in\R$
and hence it is constant and   $K$ is a  disc. \end{proof}

In fact this result can be thought as a consequence of Nitsche's result that assumes the existence of two isotopic circles, because in the case of constant width one of the isotopic circles is given at the infinity by Theorem \ref{0102b}.

\subsection{Relationship between the area of a convex set and the radius of an isotopic circle}
We compare the area of the convex set $K$ with the area enclosed by an isotopic circle of $K$.

\begin{thm}\label{0403b}Let $K$ be a compact convex set of area $F$ that has an isotopic circle $C_{\alpha}$ of radius $R$. Then 
$$F\leq F_{R} \sin^{2}(\frac{\alpha}{2}),$$
with $F_{R}=\pi R^{2}$.
\end{thm}

\begin{proof}
In \cite{CGR} it is proved  the equality
\begin{eqnarray}\label{2602}F(\alpha)\sin^{2}(\frac{\alpha}{2})=F+\frac{\pi}{4\cos^{2}(\frac{\alpha}{2})}\sum_{k\geq 2}\big(2(k^{2}+1)\cos^{2}(\frac{\alpha}{2})+g_{k}(\alpha)\big)c_{k}^{2}\end{eqnarray}
expressing  the area $F(\alpha)$ enclosed by the isotopic set $C_{\alpha}$ of a compact convex set $K$ in terms of the area $F$ of $K$, the Fourier coefficients $a_{k},b_{k}$ of the support function of $K$ ($c_{k}^{2}=a_{k}^{2}+b_{k}^{2}$), and  Hurwitz's functions $g_{k}(\alpha)$  given by
$$g_{k}(\alpha)=1+\frac{(-1)^{k}}{2}((k+1)\cos(k-1)\alpha-(k-1)\cos(k+1)\alpha).$$

It is known from \cite{green} that when $C_{\alpha}$ is a circle and $K$ is not a  disc then  $\alpha=\pi-\frac{m}{n}\pi$ with $(m,n)=1$, and $m$ odd. We will assume from now on that $K$ is not a disk.

In this case  $p(\varphi)$ is $(2\pi/n)$-periodic and so $k=\mu n$, $\mu\in\N$,  and  $g_{k}(\alpha)=1+(-1)^{\mu}\cos(\alpha)$.

So, equality \eqref{2602} says
\begin{eqnarray}\label{0403f}
&&F_{R}\sin^{2}(\frac{\alpha}{2})=\\&&F+\frac{\pi}{4\cos^{2}\frac{\alpha}{2}}\sum_{\mu}(2(\mu^{2}n^{2}-1)\cos^{2}(\frac{\alpha}{2})+1+(-1)^{\mu}\cos(\alpha))c_{n\mu}^{2}\nonumber
\end{eqnarray}
and since the coefficient 
of $c_{n\mu}^{2}$ is positive for each $\mu$, we obtain the desired inequality.

\end{proof}

\subsection{Relationship between the perimeter of a convex set and the radius of an isotopic circle}
Now we compare the perimeter of a convex set with the length of one of its isotopic circles.
\begin{thm}\label{0403c}Let $K$ be a compact convex set of perimeter  $L$ that has an isotopic circle $C_{\alpha}$ of radius $R$. Then 
$$L\leq L_{R}\sin (\frac{\alpha}{2}).$$
with $L_{R}=2\pi R$.
\end{thm}
\begin{proof}
Let us consider
the Fourier series  of the functions  $p=p(\varphi)$ and $p_{1}=p(\varphi+\pi-\alpha)$ given by 
\begin{eqnarray*}
p&=&a_{0}+\sum_{k\geq 1}a_{k} \cos(k\varphi)+b_{k}\sin(k\varphi)\\
p_{1}&=&a_{0}+\sum_{k\geq 1}A_{k} \cos(k\varphi)+B_{k}\sin(k\varphi)
\end{eqnarray*}
with
\begin{eqnarray*}
A_{k}&=&(-1)^{k+1}(-a_{k}\cos(k\alpha)+b_{k}\sin(k\alpha))\\
B_{k}&=&(-1)^{k+1}(-a_{k}\sin(k\alpha)-b_{k}\cos(k\alpha)).
\end{eqnarray*}
Then
$$\int_{0}^{2\pi}pp_{1}\,d\varphi=\int_{0}^{2\pi} \bigg(a_{0}^{2}+\sum_{k\geq 1} (a_{k}A_{k}\cos^{2}(k\varphi)+b_{k}B_{k}\sin^{2}(k\varphi))\bigg)\,d\varphi,$$
and substituting the given values of $A_{k},B_{k}$, it follows 

\begin{eqnarray}\label{2602b}\int_{0}^{2\pi}pp_{1}\,d\varphi=\frac{L^{2}}{2\pi}-\pi \sum_{k\geq 1} (-1)^{k+1}c_{k}^{2}\cos(k\alpha).\end{eqnarray}

Assuming that $C_{\alpha}$ is an isotopic circle of radius $R$, 
integrating the equality  \eqref{2202c} on this circle and taking into account  \eqref{2602b}, 
we obtain 
\begin{eqnarray*}
2L^{2}(1+\cos(\alpha))+4\pi^{2}\sum_{k\geq 1}(1+(-1)^{k}\cos(\alpha)\cos(k\alpha))c_{k}^{2}=L_{R}^{2}\sin^{2}(\alpha),
\end{eqnarray*}
where $L_{R}=2\pi R$. 

Considering, as in the previous section, that  $\alpha=\pi-\frac{m}{n}\pi$ with $(m,n)=1$ and $m$ odd, the above equation reads 

\begin{eqnarray}\label{0403e}L^{2}+2\pi^{2}\sum_{\mu, even}c_{n\mu}^{2}+2\pi^{2}\tan^{2}(\frac{\alpha}{2})\sum_{\mu, odd}c_{n\mu}^{2}=L_{R}^{2}\sin^{2}(\frac{\alpha}{2}).\end{eqnarray}

In particular we have the inequality 
$$L\leq L_{R}\sin (\frac{\alpha}{2}).$$

\end{proof}

{\bf Acknowledgement.}  The authors are grateful to A. Gasull for various conversations on the subject that have contributed to the proof
of Theorem \ref{0102b}. We also thank  E. Gallego for his useful comments.

\bibliographystyle{amsplain}

\bibliography{bibliografia}

\providecommand{\bysame}{\leavevmode\hbox to3em{\hrulefill}\thinspace}
\providecommand{\MR}{\relax\ifhmode\unskip\space\fi MR }
\providecommand{\MRhref}[2]{%
  \href{http://www.ams.org/mathscinet-getitem?mr=#1}{#2}
}
\providecommand{\href}[2]{#2}
\begin{thebibliography}{1}

\bibitem{CGR}
J.~Cuf\'{\i}, E.~Gallego, and A.~Revent\'{o}s, \emph{On the integral formulas
  of {C}rofton and {H}urwitz relative to the visual angle of a convex set},
  Mathematika \textbf{65} (2019), no.~4, 874--896.

\bibitem{green}
J.W. Green, \emph{Sets subtending a constant angle on a circle}, Duke Math. J.
  \textbf{17} (1950), 263--267.

\bibitem{Kurusa}
\'A. Kurusa, \emph{You can recognize the shape of a figure from its shadows!},
  Geom. Dedicata \textbf{59} (1996), no.~2, 113--125.

\bibitem{Nitsche}
J.~C.~C. Nitsche, \emph{Isoptic characterization of a circle (proof of a
  conjecture of {M}. {S}. {K}lamkin)}, Amer. Math. Monthly \textbf{97} (1990),
  no.~1, 45--47.

\bibitem{santalo}
L.A. Santal\'o, \emph{{Integral Geometry and Geometric Probability}}, second
  ed., Cambridge University Press, Cambridge, 2004.

\end{thebibliography}

{\sc Departament de Matemàtiques, Universitat Autònoma de Barcelona, 08193 Bellaterra, Barcelona, Catalonia. }

{\em E-mail address:} joaquim.bruna@uab.cat, julia.cufi@uab.cat,  agusti.reventos@uab.cat.

\end{document}